\documentclass[11pt]{amsart}
\usepackage{graphicx} 
\usepackage{tikz-cd}
\usepackage{todonotes}
\usepackage{amssymb}
\usepackage{amsfonts}
\RequirePackage[dvipsnames,usenames]{xcolor}
\usepackage{hyperref}
\hypersetup{colorlinks,citecolor=OliveGreen,linkcolor=Mahogany,urlcolor=Plum}
\usepackage{mabliautoref}
\input{kmacros3.sty}
\title{On a Grauert-Riemenschneider vanishing theorem in dimension 3}
\author{Rahul Ajit}
\address{Department of Mathematics,
	University of Utah,
	Salt Lake City, UT 84112, USA.}
\email{rahulajit@math.utah.edu}

\date{October 27, 2025}

\begin{document}

\begin{abstract}
    Suppose $R$ is an excellent ring of dimension $3$ and has rational singularities. Let $\pi:X \longrightarrow \Spec \ R$ be a blow-up and $\phi: W \longrightarrow X$ be any projective, birational morphism such that $X$ and $W$ are both normal, Cohen-Macaulay and have pseudorational singularities in codimension $2$. Then $\myR^{i}\phi_{*}\omega_{W}=0 \ \text{ and }\myR^{i}\pi_{*}\omega_{X} = 0$ for all $i>0$ and $X$ has rational singularities. We use this result to prove Lipman's vanishing conjecture in dimension $3$ for arbitrary characteristics and provide a few applications. 
\end{abstract}

\maketitle

\tableofcontents
	
	\newpage

\section{Introduction}
Kodaira-type vanishing theorems are fundamental tools in birational algebraic geometry. A relative variant of Kodaira vanishing is the Grauert-Riemenschneider (GR) vanishing theorem \cite{GR-Original} which states that for a projective birational morphism $f:Y\to X$ where $Y$ is a smooth projective variety over $\bC$, we have $\myR^if_*\omega_Y=0 \text{ for all } i>0.$

Rational singularities have been the gold standard for “mild” singularities in algebraic geometry. In characteristic $0$, a scheme X is said to have \textit{rational singularities} if for any resolution of singularities $f:Y\to X$, we have 
  $f_*\sO_Y \simeq \sO_X$, (i.e, $X$ is normal), and 
  $\myR^i  f_*\sO_X=0$ for $i>0$. Kempf, using Grauert-Riemenschneider vanishing, proved that a normal scheme $X$ has rational singularities if and only if it is Cohen-Macaulay and for any resolution of singularities $f:Y\to X$, we have $f_*\omega_Y\simeq \omega_X$, see \cite[page 50]{KempfToroidalEmbeddings}. There is a close relationship between MMP singularities and rational singularities; in particular, Kawamata log terminal singularities are rational (\cite{Elkik-Canonical-Rationa}) and Gorenstein rational singularities are canonical.

Unfortunately, Kodaira vanishing fails in positive characteristics (\cite{raynaud_contre-exemple_1978}), and as a result, by the cone construction one sees that, the GR vanishing fails when the dimension is at least 3 (\cite[Example 3.11]{HaconKovacsGenericVanishingFails}). On a positive note, Chatzistamatiou and R\"ulling proved in \cite{CR-Excellent-GR} that GR vanishing holds for excellent, regular schemes $X$ and $Y$ in any dimension. In dimension $3$, positive and mixed characteristics $(p > 5)$, GR vanishing holds for klt singularities, proving that they are rational, see \cite{arvidson_vanishing,  bernasconi_kollar_vanishing, bhutani2025kawamataviehwegvanishingsurfacesdel, hacon_witaszek_rationality}. 
 
However, as the resolution of singularities is not known in dimensions greater than $3$ in positive and mixed characteristics, what should be the ``correct" definition of rational singularity remains an interesting question, see \cite{kovacs_rational_2017, lyu2022propertiesbirationalderivedsplinters, ishii2025vanishinghigherdirectimages}. Due to the example in \cite[Page 174]{Cutkosky-New-Rational}, we can not simply assume $Y$ to be normal, even when $X$ is regular. Furthermore, as the statement involves the dualizing sheaf $\w_Y$ and not the dualizing complex $\w_Y^\mydot$, we need to assume $Y$ is Cohen-Macaulay. But, by applying Brodmann's Macaulayfication (\cite[Corollary 1.4]{Brodmann-Macaul}) of Cutkosky's example, Ma showed that there exists a projective, birational morphism $f:Y\to X= \bA^3_\bC$ with $Y$ arithmetically Cohen-Macaulay, such that $\myR^2  f_*\sO_Y \neq 0$, see \cite[Example 3.2]{ishii2025vanishinghigherdirectimages} for details. However, we show in our main theorem that such (counter-)examples can not occur if we assume $Y$ to be Cohen-Macaulay and pseudorational in codimension $2$.

\begin{ThA*}(cf. Theorems \ref{thm:main-vanishing}, \ref{thm:stability-modification} and, \ref{thm:modificationandrationality})
   \textit{Let $(R, \maxm)$ be an excellent, local ring of dimension 3 with rational singularities. Let $\pi:X \longrightarrow \Spec \ R$ be a blow-up where $X$ is normal, Cohen-Macaulay, and has pseudorational singularities in codimension 2. Let $\phi: W\to X$ be any projective, birational morphism where $W$ is normal, Cohen-Macaulay, and pseudorational in codimension 2. Then we have the following:
   \begin{enumerate}
       \item $\myR^{i}\phi_{*}\omega_{W}=0$ for $i>0$,
       \item $\myR^{i}\pi_{*}\omega_{X}=0$ for $i>0$,
       \item $X$ has rational singularities.
   \end{enumerate}
    }
\end{ThA*}

Using this theorem, we prove Lipman's vanishing conjecture (see \cite[Vanishing Conjecture (2.2)]{LipmanAdjointsOfIdealsInRegularLocal}) for arbitrary characteristics in dimension $3$. 

\begin{ThB*}(see \cite[(2.2), Theorem A3.]{LipmanAdjointsOfIdealsInRegularLocal}, \cite[2.7]{HyryVillamayorBriansconSkodaForIsolated} and Proposition \ref{prop:lipman-conjecture-dim3})
\textit{Let $(R, \frm)$ be a three-dimensional rational singularity, and let $\cI \subset R$ be an ideal. Consider the blow-up 
\[
f: X = \Proj \ S = \Proj\left(\bigoplus_{n \geq 0} \cI^n\right) \longrightarrow \Spec R
\] 
along $\cI$. Assume that $S$ is Cohen-Macaulay, $X$ is normal, and has pseudorational singularities in codimension 2. Then for all integers $n \geq 0$ and all positive integers $i > 0$, we have
\[
H^i(X, \cI^n \omega_X) = 0.
\]}
\end{ThB*}

\noindent In characteristic $0$, this was proved by Cutkosky (see \cite[Theorem A3.]{LipmanAdjointsOfIdealsInRegularLocal}) and by Hyry and Villamayor U. (see, \cite[2.7]{HyryVillamayorBriansconSkodaForIsolated}) in all dimensions.

\noindent We end this short note by presenting some applications of Theorem A and Theorem B, see Corollary \ref{hara-watanabe-yoshida}, Proposition \ref{hyry-rationality} and Corollary \ref{skoda}.

\section{Background}

\begin{definition}(\cite[Definition 2.76]{KollarSingularities}, \cite{lipman_rational_1969})\label{def:rational}
A local ring $(R, \mathfrak{m})$ is said to have \emph{rational singularities} if 
\begin{enumerate}
\item $R$ is normal, excellent, Cohen-Macaulay, admitting a dualising complex, and
\item there exists a resolution of singularities $f: Y \to \Spec R$ such that $\myR^i f_* \mathcal{O}_Y = 0$ for all $i > 0$.
\end{enumerate}
\end{definition}

\begin{remark}
    Thanks to \cite{CR-Excellent-GR}, we have a well-defined notion of rational singularities in arbitrary characteristics, assuming (log) resolution of singularities (which exists, up to dimension 3, see \cite{Cossart-Piltant-charP-1}, \cite{Cossart-Piltant-charP-2} and \cite{RoS-arithmetic}).
\end{remark}

\noindent A resolution-free definition of rational singularities was introduced by Lipman and
Teissier.
\begin{definition}(\cite[Section 2]{LipmanTeissierPseudorationallocalringsandatheoremofBrianconSkoda})
A local ring $(R, \mathfrak{m})$ is said to have \emph{pseudorational singularities} if
\begin{enumerate}
\item $R$ is normal and Cohen-Macaulay and, $\widehat{R}$ is reduced.
\item For every proper birational morphism $\pi: W \to \Spec R$ with $W$ normal, if $E = \pi^{-1}(\frm)$ is the closed fibre, then the canonical map 
\[
\delta_\pi: H_\frm^d(R) \rtarr H_E^d(\cO_W)
\]is injective.
\end{enumerate}
\end{definition}

\begin{remark}
    Kempf, using Grauert-Riemenschneider vanishing showed that in characteristic $0$, a normal scheme $X$ has rational singularities if and only if it has pseudorational singularities, see \cite[page 50]{KempfToroidalEmbeddings}. Also, any regular ring is pseudorational (\cite{LipmanTeissierPseudorationallocalringsandatheoremofBrianconSkoda}).
\end{remark}

\begin{definition}
   A scheme $X$ has \textit{pseudorational singularities in codimension 2} if, for all $x\in X$ with $\dim \cO_{X, x} = 2$, the ring $\cO_{X,x}$ has pseudorational (or equivalently, rational, by \cite[Page 103, Example (a)]{LipmanTeissierPseudorationallocalringsandatheoremofBrianconSkoda}) singularities.
\end{definition}
\begin{definition}(\cite[Definition 2]{BlickleMultiplierIdealsAndModulesOnToric}, \cite[Definition 2.1]{WeakOrdanityConj})
    Suppose that $X$ is either a normal, excellent $\Q$-scheme, or a normal, excellent scheme of dimension $\leq 3$ and $\Delta$ is an effective $\bQ$-Cartier divisor, $\fra$ is an ideal sheaf and $\la \geq 0$ is a real number.
Let $\pi : Y \to X$ be a proper birational morphism with $Y$ normal such that $\fra \cdot \cO_Y = \cO_Y(-G)$ is invertible, and we assume that $K_X$ and $K_Y$ agree wherever $\pi$ is an isomorphism. We assume that $\pi$ is a log resolution of $(X, \Delta, \fra)$.  Then we define the \textit{multiplier module} to be
\[
\cJ(\omega_X, \Delta, \fra^\la) = \pi_* \cO_Y(\lceil K_Y - \pi^*\Delta  - \la G\rceil) \subseteq \omega_X.
\]
When $\Delta = 0$, we omit it, i.e, we write $\cJ(\omega_X, \fra^\la) := \cJ(\omega_X, 0, \fra^\la)$.
\end{definition}
\noindent It is a general fact that these definitions are independent of the (log) resolution chosen, see \cite[Theorem 9.2.18]{LazarsfeldPositivity2} and \cite{WeakOrdanityConj}.

\begin{remark}
    Suppose that $X$ is a reduced, equidimensional, and Cohen-Macaulay scheme and $\fra$ is an ideal sheaf on $X$.  Then the pair $(X, \fra^\la)$ has rational singularities in the sense of \cite[Definition 3.1]{RationalSingPair} if and only if the multiplier submodule $\cJ(\omega_X, \fra^\la)$ is equal to $\omega_X$, see \cite[Corollary 3.8]{RationalSingPair}.
\end{remark}

We recall the Sancho de Salas exact sequence (\cite{SanchoDeSalas}) for blow ups, which serves as a bridge connecting local cohomology to sheaf cohomology.

\begin{theorem}(\cite{SanchoDeSalas, LipmanCohenMacaulaynessInGradedAlgebras, Hyry-Smith-Kawamata})\label{cor:sds-special}
The Sancho de Salas exact sequence for the Rees Algebra $S$ is 
$$
\dots \longrightarrow H^i_{\maxm_S}(S) \longrightarrow  \bigoplus_{n\in \bN} 
H^i_{\maxm}(\fra^n) \longrightarrow 
\bigoplus_{n \in \bZ} H^i_{Z}(X, \cO_X(-nE)) \longrightarrow
H^{i+1}_{\maxm_S}(S) \longrightarrow \dots,
$$
where $\mathfrak{m}_S = \mathfrak{m} \oplus \bigoplus_{n \ge 1} \mathfrak{a}^n$, $ X = \Proj \ S \to  \Spec \ R$ is the blowup along $\mathfrak{a}$ so that $\mathfrak{a} \cO_X= \cO_X(-E)$ and $Z = X \times_{\Spec R} \Spec (R/\maxm)$ is the scheme-theoretic fiber over the closed point $\maxm$ of 
$\Spec A$.
We get, as graded $R$-modules, 
$H^{d+1}_{\maxm_S}(S) \cong \oplus_{n < 0}H^d_Z(X, \cO_X(-nE))$ \cite[2.5.2 (1)]{Hyry-Smith-Kawamata}. This is because 
the maps $H^d_m(\fra^n) \longrightarrow H_Z^d(X, \cO_X(-nE))$ are surjective for all
 $n \geq 0 $
(see, \cite[page 103]{LipmanTeissierPseudorationallocalringsandatheoremofBrianconSkoda}). Hence, Matlis duality gives, \[\omega_\rS \simeq\bigoplus_{n>0}H^0(X,\omega_X(-nE)).\]
\end{theorem}

A crucial ingredient in proving our main results is the following vanishing theorem of Lodh; and Ishii and Yoshida:

\begin{theorem}[\cite{ishii2025vanishinghigherdirectimages}, \cite{RemiLodh}]\label{thm:ishii-yoshida}
Let $Y$ be a Noetherian scheme of dimension $N \ge 1$. Assume, $Y$ is locally quasi-unmixed, has pseudorational singularities in codimension two, and satisfies Serre's condition $S_3$. Let $\varphi: X \to Y$ be a proper birational morphism of finite type with $X$ normal. Then,
\[
\myR^1 \varphi_* \mathcal{O}_X = 0.
\]
\end{theorem}
 Lipman, by the following theorem, related $\Spec \ R$ having rational singularities with resolution being arithmetically Cohen-Macaulay in \cite[Theorem (4.1)]{LipmanCohenMacaulaynessInGradedAlgebras}.

\begin{theorem}(\cite[Theorem 4.1]{LipmanCohenMacaulaynessInGradedAlgebras})\label{thm:lipman-cm}
Let $F$ be a filtration on a local ring $(R, \mathfrak{m})$ with Rees algebra $R_F$ and $X = \Proj(R_F)$. Suppose:
\begin{enumerate}
\item $X$ is Cohen-Macaulay.
\item The natural map $R \to \mathbf{R}\Gamma(X, \mathcal{O}_X)$ is an isomorphism (i.e., $H^0(X, \mathcal{O}_X) = R$ and $H^i(X, \mathcal{O}_X) = 0$ for $i > 0$).
\end{enumerate}
Then for some $e > 0$, the Veronese subalgebra $R_{F^{(e)}}$ is Cohen-Macaulay.
The converse also holds: if $R_F$ is Cohen-Macaulay, then $X$ is Cohen-Macaulay and $R \xrightarrow{\sim} \mathbf{R}\Gamma(X, \mathcal{O}_X)$.
\end{theorem}

\section{Main Theorems}

We now prove our main results.

\begin{theorem}\label{thm:main-vanishing}
Let $(R, \maxm)$ be an excellent, local ring of dimension 3 that has rational singularities. Let $X = \Proj S \xrightarrow{\pi} \Spec R$ be a blow-up, where $S = \bigoplus_{n \ge 0} \mathfrak{a}^n$ for an ideal $\mathfrak{a} \subset R$. Assume that $X$ is Cohen-Macaulay, normal, and has pseudorational singularities in codimension 2.

Then
\[
\myR^i \pi_* \omega_X = 0 \quad \text{for all } i > 0.
\]
\end{theorem}

\begin{proof}
We follow the strategy of \cite{Koley-Kummini21} and proceed through six steps.

\medskip\noindent
\textbf{Step 1:} Since $R$ has rational singularities and is of dimension 3, there exists a dominating resolution of singularities $\phi: Y \to X$.

\[
\begin{tikzcd}
& Y \arrow[dl, "\phi"'] \arrow[dr, "f"] & \\
X \arrow[rr, "\pi"'] & & \Spec R
\end{tikzcd}
\]
By definition of rational singularities
\[
\myR^i (\pi \circ \phi)_* \mathcal{O}_Y = H^i(Y, \mathcal{O}_Y) = 0 \quad \text{for all } i > 0.
\]

\medskip\noindent
\textbf{Step 2:} We have two key properties of the resolution $\phi: Y \to X$:

\textbf{(a)} Since $\phi$ is a resolution and $X$ is normal, we have $\phi_* \mathcal{O}_Y = \mathcal{O}_X$.

\textbf{(b)} We prove that $\myR^i \phi_* \mathcal{O}_Y$ has zero-dimensional support for $i > 0$, if they are non-zero.
Let $x \in X$ be a point with $\dim \mathcal{O}_{X,x} = 2$. By assumption, $X$ has pseudorational singularities in codimension 2. By \cite[Page 103, Example (a)]{LipmanTeissierPseudorationallocalringsandatheoremofBrianconSkoda}, we have that pseudorational implies rational in dimension 2.

Since $\phi$ is proper and $\Spec \ \cO_{X,x} \to X$ is flat, by flat base change \cite{Hartshorne},
\[
(\myR^i \phi_* \mathcal{O}_Y)_x \cong \myR^i g_* \mathcal{O}_{Y_x}.
\]
where $g: Y_x := Y \times_X \Spec \ \mathcal{O}_{X,x} \to \operatorname{Spec} \mathcal{O}_{X,x}$ is a resolution of a two-dimensional rational singularity. So $\myR^i g_* \mathcal{O}_{Y_x} = 0$ for $i > 0$, and hence, $(\myR^i \phi_* \mathcal{O}_Y)_x = 0$ for all $x$ with $\dim \mathcal{O}_{X,x} = 2$.

Therefore, $\Supp(\myR^i \phi_* \mathcal{O}_Y) \subseteq \{x \in X : \dim \mathcal{O}_{X,x} \geq 3\}$. Since $X$ is three-dimensional and excellent, this set is finite, and hence zero-dimensional.

\medskip\noindent
\textbf{Step 3:} Consider the Leray spectral sequence for the composition $\pi \circ \phi$:
\[
E_2^{p,q} = \myR^q \pi_* (\myR^p \phi_* \mathcal{O}_Y) \implies \myR^{p+q} (\pi \circ \phi)_* \mathcal{O}_Y = H^{p+q}(Y, \mathcal{O}_Y).
\]

From Step 2(b), we have that for $p > 0$, $\myR^p \phi_* \mathcal{O}_Y$ has zero-dimensional support. Therefore, $H^i(X, \myR^p \phi_* \mathcal{O}_Y) = 0$ for $i > 0$.
Thus $$E_2^{p,q} = \myR^q \pi_* (\myR^p \phi_* \mathcal{O}_Y) = H^q(X, \myR^p \phi_* \mathcal{O}_Y) = 0$$ for $p > 0$ and $q > 0$.

The five-term exact sequence (see \cite[Proposition F.105]{Gortz-Wedhorn-2}) from this spectral sequence is
\[
0 \to \myR^1 \pi_* (\phi_* \mathcal{O}_Y) \to \myR^1 (\pi \circ \phi)_* \mathcal{O}_Y \to \pi_* (\myR^1 \phi_* \mathcal{O}_Y) \to \myR^2 \pi_* (\phi_* \mathcal{O}_Y) \to \myR^2 (\pi \circ \phi)_* \mathcal{O}_Y.
\]

Using
\begin{itemize}
\item $\phi_* \mathcal{O}_Y = \mathcal{O}_X$ (from Step 2(a)),
\item $\myR^i (\pi \circ \phi)_* \mathcal{O}_Y = 0$ for $i = 1,2$ (from Step 1),
\item $\pi_* (\myR^1 \phi_* \mathcal{O}_Y) = H^0(X, \myR^1 \phi_* \mathcal{O}_Y)$,
\end{itemize}
we obtain
\[
0 \to H^1(X, \mathcal{O}_X) \to 0 \to H^0(X, \myR^1 \phi_* \mathcal{O}_Y) \to H^2(X, \mathcal{O}_X) \to 0.
\]
We immediately get $H^1(X, \mathcal{O}_X) = 0$, and $H^0(X, \myR^1 \phi_* \mathcal{O}_Y) \cong H^2(X, \mathcal{O}_X)$.

Now, by Theorem \ref{thm:ishii-yoshida} (\cite{ishii2025vanishinghigherdirectimages}), we have $\myR^1 \phi_* \mathcal{O}_Y = 0$, implying $H^2(X, \mathcal{O}_X)=0$ as well. Note that we already have $H^3(X, \mathcal{O}_X)=0$ as the closed fibers of $\pi: X \rtarr \Spec \ R$ have dimension $\leq 2$, see \cite[Corollaire (4.2.2)]{EGAIII}.

\medskip\noindent
\textbf{Step 5:} After Step 4, we are in the situation of Theorem \ref{thm:lipman-cm} (\cite{LipmanCohenMacaulaynessInGradedAlgebras}). Therefore, there exists $N \gg 0$ such that the Veronese subring $S^{(N)} = \bigoplus_{n \ge 0} \mathfrak{a}^{nN}$ is Cohen-Macaulay. Since $X = \Proj S \cong \Proj S^{(N)}$, we may replace $S$ by $S^{(N)}$ and assume henceforth that $S$ itself is Cohen-Macaulay.

\medskip\noindent
\textbf{Step 6:} Let $d = \dim X = 3$. Apply the SdS sequence from Theorem \ref{cor:sds-special}:
\[
\cdots \to H^i_{\mathfrak{m}_S}(S) \to \bigoplus_{n \ge 0} H^i_{\mathfrak{m}}(\mathfrak{a}^n) \to \bigoplus_{n \in \mathbb{Z}} H^i_Z(X, \mathcal{O}_X(n)) \to H^{i+1}_{\mathfrak{m}_S}(S) \to \cdots
\]
where $\mathfrak{m}_S = \mathfrak{m} \oplus \bigoplus_{n \ge 1} \mathfrak{a}^n$ and $Z = \pi^{-1}(\mathfrak{m})$.

Since $S$ is Cohen-Macaulay of dimension $d+1 = 4$, we have $H^i_{\mathfrak{m}_S}(S) = 0$ for $i < 4$.
Thus, for $i < 3$, we have,
\[
0 \to \bigoplus_{n \ge 0} H^i_{\mathfrak{m}}(\mathfrak{a}^n) \xrightarrow{\cong} \bigoplus_{n \in \mathbb{Z}} H^i_Z(X, \mathcal{O}_X(n)) \to 0
\]
For $n = 0$ and $i<3$, $H^i_{\mathfrak{m}}(R) = 0$ (since $R$ is Cohen-Macaulay), so $H^i_Z(X, \mathcal{O}_X) = 0$.

\noindent By Grothendieck duality (see \cite[Lemma (4.2)]{LipmanCohenMacaulaynessInGradedAlgebras},
\[
H^i_Z(X, \mathcal{O}_X) \cong \operatorname{Hom}_R(H^{3-i}(X, \omega_X), E(R/\mathfrak{m}))= 0
\]
This immediately gives us
\[
H^{3-i}(X, \omega_X) = \myR^{3-i}\pi_*\omega_X = 0, \ \forall i<3.
\]

This completes the proof of Theorem \ref{thm:main-vanishing}.
\end{proof}

\subsection{Pseudorational Modification}

We now prove that the vanishing property is stable under further ``pseudorational in codimension 2" modifications.

\begin{theorem}\label{thm:stability-modification}

Let $Y \xrightarrow{\psi} X \xrightarrow{\pi} \Spec R$, where $R$ is of dimension $3$ and has rational singularities. Assume the following.
\begin{enumerate}
\item $X$ has pseudorational singularities in codimension $2$;
\item $X$ is Cohen-Macaulay;
\item $Y$ is a dominating resolution of singularities;
\item $\pi$ is a blow-up of $\Spec R$.
\end{enumerate}
Then, $X$ has rational singularities. In particular, $\myR^i \psi_* \omega_Y = 0$ for all $i > 0$. 
\end{theorem}

\begin{proof}
Since $R$ has rational singularities and $\pi \circ \psi: Y \to \Spec R$ is a resolution of singularities, we have:
\[
\myR^i (\pi \circ \psi)_* \cO_Y = 0 \quad \text{for all } i > 0.
\]
Since $\Spec R$ is affine, \cite[Proposition III.8.1]{hartshorne_algebraic_1977} implies that,
\[
H^i(Y, \cO_Y) = 0 \quad \text{for all } i > 0.
\]
By repeating Step 2 in the proof of Theorem \ref{thm:main-vanishing} we get the support of $\myR^p \psi_* \cO_Y$ is at most zero-dimensional. Thus,
\[
H^q(X, \myR^p \psi_* \cO_Y) = 0 \ \mathrm{for} \ \ p,q > 0
\]
We will show that for $p > 0$,  $\myR^p \psi_* \cO_Y = 0$ by analyzing the Leray spectral sequence of $\psi$: 
\[
E_2^{p,q} = H^q(X, \myR^p \psi_* \cO_Y) \Rightarrow H^{p+q}(Y, \cO_Y).
\]
We know that $E_2^{p,q} = 0$ for $p > 0$ and $q > 0$. On the other hand, $E_2^{0,q} = H^q(X, \psi_* \cO_Y) = H^q(X, \cO_X)$ for $p = 0$
and $E_2^{p,0} = H^0(X, \myR^p \psi_* \cO_Y)$ for $q = 0$ may be nonzero.
Note that $H^{p + q}(Y, \cO_Y) = 0$ for $p + q > 0$, so that $E_\infty^{p, q} = 0$ for all $p + q  > 0$. To show vanishing, we will check that the spectral sequence degenerates at the 2nd page. We note that the differential of $E_2$ has bidegree $(-1, 2)$, so the only possibly non-zero differential on the second page is $$E_2^{1, 0} = H^0(X, \myR^1\psi_* \cO_Y) \stackrel{d_2}\longrightarrow H^2(X, \cO_X) = E_2^{0,2}.$$ Thus, $E_\infty^{p, q} = E^{p, q}_2$ for $(p, q) \neq (1,0), (0,2)$. Hence, $$H^0(X, \myR^2_* \psi_* \cO_Y)  = E_2^{2, 0}= E_{\infty}^{2, 0} =  0 \text{ and }H^0(X, \myR^3_* \psi_* \cO_Y)  = E_2^{3, 0}= E_{\infty}^{3, 0} = 0,$$ so by Step 2 we have $\myR^2\psi_* \cO_Y = \myR^3 \psi_* \cO_Y = 0$. Lastly, since $X$ is Cohen-Macaulay, we conclude from Theorem \ref{thm:ishii-yoshida} that $\myR^1\psi_* \cO_Y = 0$. Thus, the map $\cO_X \to \myR \psi_* \cO_Y$ is a quasi-isomorphism and $X$ has rational singularities. 
By Grothendieck duality,
\[
\myR\psi_*\myR\mathcal{H}om_Y(\cO_Y, \omega_Y) \cong \myR\mathcal{H}om_X(\myR\psi_*\cO_Y, \omega_X)
\]
The left side is $\myR\psi_*\omega_Y$. The right side is $\myR\mathcal{H}om_X(\cO_X, \omega_X) = \omega_X$ since $\myR^i\psi_*\cO_Y = 0$ for $i > 0$. Therefore $\myR\psi_*\omega_Y \cong  \omega_X$ i.e, $\myR^i\psi_*\omega_Y = 0$ for all $i > 0$.
\end{proof}
As a quick application, we get the following 3-dimensional analog of a result of Hara, Watanabe and Yoshida, see \cite[Theorem 3.5]{HWY-02F-Rational}). 

\begin{corollary}(\cite[Theorem 3.5]{HWY-02F-Rational})\label{hara-watanabe-yoshida}
    Let $(R, \maxm)$ be an $F$-finite, local ring of characteristic $p$ and dimension 3 that has rational singularities and $\mathfrak{a} \subset R$ is an $\frm$-primary ideal. Suppose, the blow up along $\fra$, $X = \Proj S \xrightarrow{\pi} \Spec R$ is a resolution of singularities, where $S = \bigoplus_{n \ge 0} \mathfrak{a}^n$. Then $S^{(N)}$ is $F$-rational for all $N\gg 0.$
\end{corollary}
\begin{proof}
    From Step 5 in the proof of Theorem \ref{thm:main-vanishing} we get that $S^{(N)}$ is Cohen-Macaulay for all $N\gg 0.$ We follow \cite[Proof of Theorem 3.5]{HWY-02F-Rational} for the rest. First note that, $\forall \ 0\neq ft^N \in \fra^Nt^N$, $S^{(N)}_{ft^N}$ is regular. So by \cite{HochsterHunekeFRegularityTestElementsBaseChange}, some power of $f$ is a test element, giving $S_+^{(N)} \subseteq \sqrt{\tau(S^{(N)})}.$ Then, by \cite[Lemma 3.4]{HWY-02F-Rational}, we just need to show that the Frobenius map acts injectively on $H^4_{\frm_{S^{(N)}}}(S^{(N)})$.  Note that, by Theorem \ref{cor:sds-special}, we have $$H^{4}_{\maxm_{S^{(N)}}}(S^{(N)}) \cong \oplus_{n < 0}H^3_E(X, \cO_X(nN)) .$$ So, we need to show that for all $n<0,$ $$F: H^3_E(X, \cO_X(nN)) \longrightarrow H^3_E(X, \cO_X(pnN)) \mathrm{ \ \ is \ injective .}$$ By \cite[Proposition 3.5]{hara_characterization_1998}, injectivity follows if we can show the following two vanishings:
    \begin{enumerate}
        \item $H^j_E(X, \Omega_X^i(\log \ E)(nN)) = 0$ for $i + j = 2$ and $i>0$,
        \item $H^j_E(X, \Omega_X^i(\log \ E)(pnN)) = 0$ for $i+j = 3$ and $i>0$
    \end{enumerate}
Now note that $H^0_E$ of locally free sheaves are always $0$, and in all other cases, we win by Serre vanishing as $N \gg 0.$ So we get $S^{(N)}$ is $F$-rational for all $N\gg 0.$
\end{proof}

\begin{theorem}\label{thm:modificationandrationality}
Suppose $R$ is of dimension $3$ and has rational singularities. Let $X$ be a blow-up of $\Spec R$ and $X$ is normal, has pseudorational singularities in codimension 2 and is Cohen-Macaulay. Then for any projective, birational morphism $\phi:W\to X$ with $W$ normal, pseudorational in codimension 2 and Cohen-Macaulay, we have $\myR^{i}\phi_{*}\omega_{W}=0$ for $i>0$.
\end{theorem}

\begin{proof}
For the first statement, take a common resolution

\[
\begin{tikzcd}
& Z \arrow[dl, "\alpha"'] \arrow[dr, "\beta"] & \\
Y \arrow[dr, "\psi"'] & & W \arrow[dl, "\phi"] \\
& X &
\end{tikzcd}
\]
\noindent where $Z$ is a resolution of singularities dominating both $Y$ and $W$. By Theorem \ref{thm:stability-modification}, we have $\myR^{i}\psi_{*}\omega_{Y}=0$ and $\myR^{i}\alpha_{*}\omega_{Z}=0$. Then, again by repeatedly applying Theorem \ref{thm:stability-modification}, we get,
\[
\omega_{X}=\myR\psi_{*}\omega_{Y}=\myR\psi_{*}(\myR\alpha_{*}\omega_{Z})=\myR(\psi\circ\alpha)_{*}\omega_{Z}=\myR(\phi\circ\beta)_{*}\omega_{Z}=\myR\phi_{*}(\myR\beta_{*}\omega_{Z})=\myR\phi_{*}\omega_{W}
\]
Hence $\myR\phi_{*}\omega_{W}\simeq\omega_{X}$, so $\myR^{i}\phi_{*}\omega_{W}=0$ for $i>0$.
\end{proof}

\section{Applications}

\subsection{Lipman's Vanishing Conjecture in Dimension Three}

We now apply our Theorem \ref{thm:main-vanishing} to establish a slight generalization of Lipman's vanishing conjecture (see \cite[Vanishing Conjecture (2.2)]{LipmanAdjointsOfIdealsInRegularLocal}) for arbitrary characteristics in dimension $3$. In characteristic $0$ this was proved by Cutkosky (see \cite[Theorem A3.]{LipmanAdjointsOfIdealsInRegularLocal}) and by Hyry and Villamayor U. (see, \cite[2.7]{HyryVillamayorBriansconSkodaForIsolated}) for all dimensions. Our proof follows the same strategy as in \cite[2.7]{HyryVillamayorBriansconSkodaForIsolated}.

\begin{proposition}(see \cite[(2.2), Theorem A3.]{LipmanAdjointsOfIdealsInRegularLocal}, \cite[2.7]{HyryVillamayorBriansconSkodaForIsolated})
\label{prop:lipman-conjecture-dim3}
Let $(R, \frm)$ be a three-dimensional rational singularity, and let $\cI \subset R$ be an ideal. Consider the blow-up 
\[
f: X = \Proj \ S = \Proj\left(\bigoplus_{n \geq 0} \cI^n\right) \longrightarrow \Spec R
\] 
along $\cI$. Assume that $S$ is Cohen-Macaulay, and $X$ is normal, and has pseudorational singularities in codimension 2. Then for all integers $n \geq 0$ and all positive integers $i > 0$, we have
\[
H^i(X, \cI^n \omega_X) = 0.
\]
\end{proposition}

\begin{proof}
We follow \cite[Remark 2.8]{HyryVillamayorBriansconSkodaForIsolated} closely and use Grothendieck's natural construction as in \cite[6.2.1]{Hyry-Smith-Kawamata}.
Consider the graded $S$-algebra defined by:
\[
S^\# := S \oplus S_{\geq 1} \oplus S_{\geq 2} \oplus \cdots,
\]
where $S_{\geq n} = \bigoplus_{m \geq n} \cI^m$ for each integer $m \geq 0$. Define the scheme
\[
Y := \Proj(S^\#).
\]

The scheme $Y$ admits two geometric interpretations:

\begin{enumerate}
\item The algebra $S^\#$ is naturally isomorphic to the Rees algebra of the ideal $S_{\geq 1} \subset S$. So, the natural projection
\[
\theta: Y \longrightarrow \Spec S
\]is the blow-up of $\Spec S$ along $S_{\geq 1}$.

\item There exists a canonical isomorphism of $X$-schemes
\[
Y \cong \Spec_X \left( \bigoplus_{m \geq 0} \cO_X(m) \right).
\] This isomorphism identifies $Y$ with the total space of the line bundle $\cO_X(-1)^\vee$, which is the dual of the tautological line bundle associated with the blow-up $X = \Proj S$ and $\cO_X(1) \simeq \cI \cdot \cO_X$.  Let $\eta: Y \longrightarrow X$ denote the corresponding bundle map.
\end{enumerate}

We have the following commutative diagram
\[
\begin{tikzcd}
Y \arrow[r, "\theta"] \arrow[d, "\eta"'] & \Spec S \arrow[d, "\pi"] \\
X \arrow[r, "f"'] & \Spec R
\end{tikzcd}
\]
where $\pi: \Spec S \to \Spec R$ is the natural affine morphism induced by the inclusion $R \hookrightarrow S$.

\noindent Since $X$ has rational singularities by Theorem~\ref{thm:stability-modification} and $\eta: Y \to X$ is a smooth affine morphism, it follows that $Y$ also has rational singularities. The morphism $\eta: Y \to X$ is smooth of relative dimension 1. The relative cotangent bundle satisfies $\Omega_{Y/X} \cong \eta^* \cO_X(1)$. So, we have
\[
\omega_Y \cong \eta^* \omega_X \otimes \Omega_{Y/X} \cong \eta^* \omega_X \otimes \eta^* \cO_X(1) \cong \eta^*(\omega_X(1)).
\]
Since $\eta$ is affine
\[
\eta_* \omega_Y \cong \eta_* \left( \eta^* (\omega_X(1)) \right) \cong \omega_X(1) \otimes \eta_* \cO_Y.
\]

As $Y$ is the total space of $\cO_X(-1)$, we have
\[
\eta_* \cO_Y \cong \bigoplus_{m \geq 0} \cO_X(m).
\]
Therefore
\[
\eta_* \omega_Y \cong \bigoplus_{m \geq 0} \omega_X(1) \otimes \cO_X(m) \cong \bigoplus_{m \geq 0} \omega_X(m+1) \cong \bigoplus_{m \geq 0} \cI^{m+1} \omega_X. \tag{1}
\]

\noindent Now, if we view $Y$ as \[
\theta: Y \longrightarrow \Spec S
\] the blow-up of $\Spec S$ along $S_{\geq 1}$, then we get the associated graded ring $\bigoplus_{n \geq 0} \frac{S_{\geq n}}{S_{\geq n+1}} \simeq S$ which is Cohen-Macaulay by assumption. Also note that as $Y$ has rational singularities, by \cite[Theorem 1.4]{SanchoDeSalas}, we get that $H^i(Y, \omega_Y) = 0 \ \text{for all } i > 0.$ Since $\eta$ is affine, we get
\[
0 =H^i(Y, \omega_Y) \cong H^i(X, \eta_* \omega_Y) \cong \bigoplus_{m \geq 0} H^i(X, \cI^{m+1} \omega_X) \quad \text{for all } i \geq 0.
\]
For the case $n = 0$, the result follows from Theorem~\ref{thm:main-vanishing}. This finishes the proof.
\end{proof}
\begin{proposition}
Let $(R, \frm)$ be a three-dimensional rational singularity and $f: X \ \longrightarrow \Spec R$ be any projective, birational morphism. Assume that $X$ is normal, and has pseudorational singularities in codimension 2. Let $D \subset X$ be an $f$-ample divisor such that the section ring $R(X, D)$ is Cohen-Macaulay, and generated in degree 1. Then for all integers $n \geq 0$ and all positive integers $i > 0$, we have
\[
H^i(X,  \omega_X(nD)) = 0.
\]    
\end{proposition}
\begin{proof}
    The proof is immediate from the previous proof after using \cite[Theorem II.7.17]{Hartshorne} (and \cite[\href{https://stacks.math.columbia.edu/tag/01VK}{Tag 01VK}]{stacks-project}), see \cite[Section 1.3]{SmithVanishingSingularitiesAndEffectiveBounds} for details.
\end{proof}
\noindent We apply Proposition \ref{prop:lipman-conjecture-dim3} to get the following theorem of E. Hyry.
\begin{proposition}(\cite[Theorem 3.2]{HyryBlowUp})\label{hyry-rationality}
    Let $(R, \frm)$ be a 3-dimensional local ring with rational singularities. Let $X = \Proj S \xrightarrow{\pi} \Spec R$ be a blow-up of an ideal $\cI \subset R$, with X normal. If $\cJ(\omega_R, \cI) = \w_R$, then $X$ has rational singularities.
\end{proposition}

\begin{proof}
    The proof follows \cite[Theorem 3.2]{HyryBlowUp} verbatim by using Proposition \ref{prop:lipman-conjecture-dim3} and noting that $\cI^{d-2}=\cI$. So we do not include it here.
\end{proof}
\begin{remark}
    Hyry proved this result in \cite[Theorem 3.2]{HyryBlowUp}, for all dimensions, under the assumption that $R$ is a regular local ring, essentially of finite type over a field of characteristic 0. Our result is for any $3$-dimensional, excellent, local ring with rational singularities, as in \ref{def:rational}.
\end{remark}
\subsection{Brian\c{c}on-Skoda-type result}
Here we briefly mention another quick application of Proposition \ref{prop:lipman-conjecture-dim3}. Most probably, this was the original motivation of Lipman to frame his Vanishing Conjecture (\cite[Vanishing Conjecture (2.2)]{LipmanAdjointsOfIdealsInRegularLocal}).

\begin{corollary}(\cite[Conjecture 1.6]{LipmanAdjointsOfIdealsInRegularLocal})\label{skoda}
    Let $(R, \frm)$ be a 3-dimensional local ring with rational singularities and $\mathfrak{a}$ be an ideal of $R$ with analytic spread $l$. Then, $$\cJ(\omega_R, \fra^{n+1}) = \fra \cdot \cJ(\omega_R, \fra^{n}) \ \mathrm{for \ all \ } n \geq l-1.$$
\end{corollary}
\begin{proof}
    The exact proof in \cite[(2.3)]{LipmanAdjointsOfIdealsInRegularLocal} works verbatim.
\end{proof}

\begin{que}
It would be interesting know if one can prove  \cite[Theorem 2.9 and Theorem 2.12]{HyryVillamayorBriansconSkodaForIsolated} in dimension 3 for arbitrary characteristics.
\end{que}

\section{Acknowledgment}
I would like to thank my advisors, Christopher Hacon and Karl Schwede, for their constant encouragement, unwavering support, inspiring teachings, and infinite patience. I am grateful to Manoj Kummini for extremely helpful discussions about \cite{Koley-Kummini21}, which played a crucial role in this note. Finally, I would like to thank Daniel Apsley, Shikha Bhutani, Harold Blum, Mircea \mustata, and Joseph Sullivan for providing valuable feedback which improved the exposition. I was partially supported by NSF research grant DMS-2301374 and by a grant from the Simons Foundation SFI-MPS-MOV-00006719-07 while working on this project.\\

\bibliographystyle{skalpha}
\bibliography{MainBib}
\end{document}